\documentclass{ws}

\begin{document}

\markboth{D. Mekerov \& M. Manev}{Natural Connection with Totally
Skew-Symmetric Torsion}

%
%

\title{NATURAL CONNECTION WITH TOTALLY SKEW-SYMMETRIC TORSION ON
RIEMANNIAN ALMOST PRODUCT MANIFOLDS
}

\author{DIMITAR MEKEROV$^\ast$ and MANCHO MANEV$^\dag$}

\address{Faculty of Mathematics and Informatics,
Plovdiv University\\ 236 Bulgaria Blvd, Plovdiv, 4003, Bulgaria\\
$^\ast$\email{mircho@uni-plovdiv.bg}
$^\dag$\email{mmanev@uni-plovdiv.bg} }

\newcommand{\ie}{\emph{i.e. }}
\newcommand{\eg}{\emph{e.g. }}
\newcommand{\X}{\mathfrak{X}}
\newcommand{\W}{\mathcal{W}}
\newcommand{\s}{\mathfrak{S}}
\newcommand{\g}{\mathfrak{g}}
\newcommand{\R}{\mathbb{R}}
\newcommand{\n}{\nabla}
\newcommand{\tr}{\mathrm{tr}}
\newcommand{\lm}{\lambda}
\newcommand{\ep}{\varepsilon}
\newcommand{\dd}{\mathrm{d}}

\newcommand{\norm}[1]{\left\Vert#1\right\Vert ^2}
\newcommand{\nP}{\norm{\nabla P}}

\newcommand{\thmref}[1]{The\-o\-rem~\ref{#1}}
\newcommand{\propref}[1]{Pro\-po\-si\-ti\-on~\ref{#1}}
\newcommand{\secref}[1]{Section~\ref{#1}}
\newcommand{\lemref}[1]{Lem\-ma~\ref{#1}}
\newcommand{\coref}[1]{Corollary~\ref{#1}}

\maketitle


\begin{abstract}
On a Riemannian almost product manifold $(M,P,g)$ we consider a
linear connection preserving the almost product structure $P$ and
the Riemannian metric $g$ and having a totally skew-symmetric
torsion. We determine the class of the manifolds $(M,P,g)$
admitting such a connection and prove that this connection is
unique in terms of the covariant derivative of $P$ with respect to
the Levi-Civita connection. We find a necessary and sufficient
condition the curvature tensor of the considered connection to
have similar properties like the ones of the K\"ahler tensor in
Hermitian geometry. We pay attention to the case when the torsion
of the connection is parallel. We consider this connection on a
Riemannian almost product manifold $(G,P,g)$ constructed by a Lie
group $G$.
\end{abstract}


\ccode{2010 Mathematics Subject Classification: 53C05, 53C15,
53C25, 53B05, 22E60}


\section*{Introduction}\label{sec-intro}

There is a strong interest in the metric connections with totally
skew-symmet\-ric torsion (3-form). These connections arise in a
natural way in theoretical and mathematical physics. For example,
such a connection is of particular interest in string theory
\cite{Stro}. In mathematics this connection was used by Bismut to
prove a local index theorem for non-K\"ahler Hermitian manifolds
\cite{Bis}. Such a connection is known as a \emph{KT-connection}
(K\"ahler with torsion) or a \emph{Bismut connection} on an almost
Hermitian manifold. The KT-geometry is a natural generalization of
the K\"ahler geometry, since when the torsion is zero the
KT-connection coincides with the Levi-Civita connection.
According
to Gauduchon \cite{Gaud}, on any Hermitian manifold, there exists
a unique connection preserving the almost complex structure and
the metric, whose torsion is totally skew-symmetric.

There are lots of physical applications involving a Riemannian
almost product manifold $(M,P,g)$, \ie a differentiable manifold
$M$ with an almost product structure $P$ and a Riemannian metric
$g$. The goal of the present work is to study the linear
connections on $(M,P,g)$ with totally skew-symmetric torsion
preserving the structures $P$ and $g$. We use the notions of a
\emph{Riemannian $P$-manifold}, a \emph{Riemannian $P$-tensor} and
a \emph{RPT-connection} (Riemannian $P$- with torsion) as
analogues of the notions of a K\"ahler tensor, a K\"ahler manifold
and a KT-connection  in Hermitian geometry. The point in question
in this work is the determination of the class of manifolds
admitting a RPT-connection, the uniqueness of the RPT-connection
on such manifolds and the investigation of an example of such a
connection.

The present paper is organized as follows. In Sec.~\ref{sec-prel}
we give necessary facts about Riemannian almost product manifolds.
We recall facts about natural connections (\ie connections
preserving $P$ and $g$) with torsion on Riemannian almost product
manifolds.
The main results are \thmref{thm-3.2} and \thmref{thm-3.4} in
Sec.~\ref{sec-KT}. We prove that a RPT-connection exists only on a
Riemannian almost product manifold in the class $\W_3$ from the
classification in \cite{Sta-Gri}. This classification is made
regarding the covariant derivatives of $P$ with respect to the
Levi-Civita connection $\n$. We establish the presence of a unique
RPT-connection $\n'$ which torsion is expressed by $\n P$. We find
a relation between the scalar curvatures for $\n$ and $\n'$ and
prove that they are equal if and only if $(M,P,g)$ is a Riemannian
$P$-manifold.
In Sec.~\ref{sec-Ptensor} we obtain a relation between the
curvature tensors of $\n$ and $\n'$. It is a necessary and
sufficient condition the curvature tensor of $\n'$ to be a
Riemannian $P$-tensor.
In Sec.~\ref{sec-parT} we consider the case when the torsion of
$\n'$ is parallel and its curvature tensor is a Riemannian
$P$-tensor.
In Sec.~\ref{sec-exa} we study the RPT-connection $\n'$ on an
example of a Riemannian almost product $\W_3$-manifold $(G,P,g)$
constructed in \cite{MekDobr} by a Lie group $G$.


\section{Preliminaries}\label{sec-prel}

Let $(M,P,g)$ be a \emph{Riemannian almost product manifold}, \ie
a differentiable manifold $M$ with a tensor field $P$ of type
$(1,1)$ and a Riemannian metric $g$ such that
\begin{equation*}\label{1}
    P^2x=x,\quad g(Px,Py)=g(x,y)
\end{equation*}
for arbitrary $x$, $y$ of the algebra $\X(M)$ of the smooth vector
fields on $M$. Obviously $g(Px,y)=g(x,Py)$  is valid.

Further $x,y,z,w$ will stand for arbitrary elements of $\X(M)$ or
vectors in the tangent space $T_pM$ at $p\in M$.

In this work we consider Riemannian almost product manifolds with
$\tr{P}=0$. In this case $(M,P,g)$ is an even-dimensional
manifold. If $\dim{M}=2n$ then the \emph{associated metric}
$\tilde{g}$ of $g$, determined by $\tilde{g}(x,y)=g(x,Py)$, is an
indefinite metric of signature $(n,n)$.

The classification in \cite{Sta-Gri} of Riemannian almost product
manifolds is made with respect to the tensor $F$ of type (0,3),
defined by
\begin{equation}\label{2}
F(x,y,z)=g\left(\left(\nabla_x P\right)y,z\right),
\end{equation}
where $\nabla$ is the Levi-Civita connection of $g$. The tensor $F$ has the following properties:
\begin{equation}\label{3}
    F(x,y,z)=F(x,z,y)=-F(x,Py,Pz),\quad F(x,y,Pz)=-F(x,Py,z).
\end{equation}
The basic classes of the classification in \cite{Sta-Gri} are
$\W_1$, $\W_2$ and $\W_3$. Their intersection is the class $\W_0$,
determined by the condition $F=0$ or equivalently $\n P=0$. A
manifold $(M,P,g)$ with  $\n P=0$ is called a \emph{Riemannian
$P$-manifold} in \cite{Sta87}. In this classification there are
included the classes $\W_1\oplus\W_2$, $\W_1\oplus\W_3$,
$\W_2\oplus\W_3$ and the class $\W_1\oplus\W_2\oplus\W_3$ of all
Riemannian almost product manifolds.

In the present work we consider manifolds of the class $\W_3$.
This class is determined by the condition
\begin{equation}\label{4}
    \mathop{\s}_{x,y,z} F(x,y,z)=0,
\end{equation}
where $\mathop{\s}_{x,y,z}$ is the cyclic sum by $x, y, z$. This
is the only class of the basic classes $\W_1$, $\W_2$ and $\W_3$,
where each manifold (which is not a Riemannian $P$-manifold) has a
non-integrable almost product structure $P$. This means that in
$\W_3$ the Nijenhuis tensor $N$, determined by
\begin{equation*}\label{4'}
    N(x,y)=\left(\nabla_x P\right)Py-\left(\nabla_y P\right)Px
    +\left(\nabla_{Px} P\right)y-\left(\nabla_{Py} P\right)x,
\end{equation*}
is non-zero.

Further, the manifolds of the class $\W_3$ we call
\emph{Riemannian almost product  $\W_3$-mani\-folds}.

A tensor $L$ of type (0,4) with pro\-per\-ties%
\begin{gather}
L(x,y,z,w)=-L(y,x,z,w)=-L(x,y,w,z),\label{9}\\
\mathop{\s} \limits_{x,y,z} L(x,y,z,w)=0 
,\label{10} \\
L(x,y,Pz,Pw)=L(x,y,z,w)\label{11}
\end{gather}
is called a \emph{Riemannian $P$-tensor} \cite{Mek8}.

As it is known the curvature tensor $R$ of a Riemannian manifold
with metric $g$ is determined by $
    R(x,y)z=\nabla_x \nabla_y z - \nabla_y \nabla_x z -
    \nabla_{[x,y]}z
$ and the corresponding tensor of type $(0,4)$ is defined as
follows $
    R(x,y,z,w)=g(R(x,y)z,w).
$ If the curvature tensor $R$ on a Riemannian manifold $(M,P,g)$
is a Riemannian $P$-tensor, then $(M,P,g)$ is a Riemannian
$P$-manifold.


Let the components of the inverse matrix of $g$ with respect to
the basis  $\{e_i\}$ of $T_pM$ be $g^{ij}$. Then the quantities
$\rho$ and $\tau$, determined by $\rho(y,z)=g^{ij}R(e_i,y,z,e_j)$
and $\tau=g^{ij}\rho(e_i,e_j)$, are the Ricci tensor and the
scalar curvature for $\n$, respectively.

The \emph{square norm} of $\nabla P$ is defined by
\begin{equation}\label{7''}
\nP=g^{ij}g^{ks}g\left(\left(\nabla_{e_i}P\right)e_k,\left(\nabla_{e_j}P\right)e_s\right).
\end{equation}
Obviously, $\nP=0$ is valid if and only if $(M,P,g)$ is a
Riemannian $P$-manifold.

Let $\n'$ be a linear connection with a tensor $Q$ of the
transformation $\n \rightarrow\n'$ and a torsion $T$, \ie
\begin{equation}\label{7'}
\n'_x y=\n_x y+Q(x,y),\quad T(x,y)=\n'_x y-\n'_y x-[x,y].
\end{equation}
The corresponding (0,3)-tensors are defined by
\begin{equation*}\label{12}
    Q(x,y,z)=g(Q(x,y),z), \qquad T(x,y,z)=g(T(x,y),z).
\end{equation*}
The symmetry of the Levi-Civita connection implies
\[
    T(x,y)=Q(x,y)-Q(y,x),\qquad
    T(x,y)=-T(y,x).
\]

A  decomposition of the space $\mathcal{T}$ of the torsion tensors
$T$ of type (0,3) (\ie $T(x,y,z)\allowbreak=-T(y,x,z)$) is valid
on a Riemannian  almost product manifold $(M,P,g)$:
$\mathcal{T}=\mathcal{T}_1\oplus\mathcal{T}_2\oplus\mathcal{T}_3\oplus\mathcal{T}_4$,
where $\mathcal{T}_i$ $(i=1,2,3,4)$ are the invariant orthogonal
subspaces \cite{Mi}. The projection operators $p_i$ of
$\mathcal{T}$ in $\mathcal{T}_i$ are determined as follows:
\begin{equation}\label{14'}
  \begin{split}
&
    p_1(x,y,z)=\frac{1}{8}\bigl\{2T(x,y,z)-T(y,z,x)-T(z,x,y)-T(Pz,x,Py)\bigr.\\
& \phantom{p_1(x,y,z)=\frac{1}{8}}
    +T(Py,z,Px)+T(z,Px,Py)-2T(Px,Py,z)\\
& \phantom{p_1(x,y,z)=\frac{1}{8}}
    +T(Py,Pz,x)+T(Pz,Px,y)-T(y,Pz,Px)\bigr\},\\
&
    p_2(x,y,z)=\frac{1}{8}\bigl\{2T(x,y,z)+T(y,z,x)+T(z,x,y)+T(Pz,x,Py)\bigr.\\
&  \phantom{p_2(x,y,z)=\frac{1}{8}}
    -T(Py,z,Px)-T(z,Px,Py)-2T(Px,Py,z)\\
& \phantom{p_2(x,y,z)=\frac{1}{8}}
    -T(Py,Pz,x)-T(Pz,Px,y)+T(y,Pz,Px)\bigr\},\\
  &
    p_3(x,y,z)=\frac{1}{4}\bigl\{T(x,y,z)+T(Px,Py,z)-T(Px,y,Pz)-T(x,Py,Pz)\bigr\},\\
  &
    p_4(x,y,z)=\frac{1}{4}\bigl\{T(x,y,z)+T(Px,Py,z)+T(Px,y,Pz)+T(x,Py,Pz)\bigr\}.
  \end{split}
\end{equation}

\begin{definition}[\cite{Mi}]\label{defn-2.1}
A linear connection $\n'$ on a Riemannian almost product manifold
$(M,P,g)$ is called a \emph{natural connection} if $\n' P=\n'
g=0$.
\end{definition}

If $\n'$ is a linear connection with a tensor $Q$ of the
transformation $\n \rightarrow\n'$ on a Riemannian almost product
manifold, then it is  a natural connection if and only if the
following conditions are satisfied \cite{Mi}:
\begin{gather}
    F(x,y,z)=Q(x,y,Pz)-Q(x,Py,z),\label{15}\\
    Q(x,y,z)=-Q(x,z,y).\label{16}\nonumber
\end{gather}

Let $\Phi$ be the (0,3)-tensor determined by
\begin{equation*}\label{17}
    \Phi(x,y,z)=g\left(\widetilde{\nabla}_x y - \n_x y,z \right),
\end{equation*}
where $\widetilde{\nabla}$ is the Levi-Civita connection of the
associated metric $\tilde{g}$.

\begin{theorem}[\cite{Mi}]\label{t-3.1}
A linear connection with torsion $T$ on a Riemannian almost
product manifold $(M,P,g)$ is natural if and only if
  \begin{eqnarray}
    4p_1(x,y,z)=&-&\Phi(x,y,z)+\Phi(y,z,x)-\Phi(x,Py,Pz)\nonumber\\ \label{18}
    &-&\Phi(y,Pz,Px)+2\Phi(z,Px,Py),\nonumber
\end{eqnarray}
\[
4p_3(x,y,z)=-g(N(x,y),z)=-2\left\{\Phi(z,Px,Py)+\Phi(z,x,y)\right\}.\nonumber\label{19}
\]
\hfill$\Box$
\end{theorem}

\begin{theorem}[\cite{MekDobr}]\label{thm-2.2}
    For the torsion $T$ of a natural connection on a Riemannian $\W_3$-manifold
    $(M,P,g)\notin\W_0$ the following properties
    hold
    \begin{equation*}\label{21'}
        p_1=0,\qquad p_3\neq 0.
    \end{equation*}\hfill$\Box$
\end{theorem}

\section{RPT-Connection on a Riemannian Almost Product Manifold}\label{sec-KT}


Let $\n'$ be a metric connection with totally skew-symmetric
torsion $T$ on a Riemannian manifold $(M,g)$. Since $T$ is a
3-form, \ie
\begin{equation*}\label{22}
    T(x,y,z)=-T(y,x,z)=-T(x,z,y)=-T(z,y,x),
\end{equation*}
it is valid
\begin{equation}\label{23}
    Q(x,y,z)=\frac{1}{2}T(x,y,z).
\end{equation}
Then we have
\[
    g\left(\n'_xy,z\right)=g\left(\n_xy,z\right)+\frac{1}{2}T(x,y,z).
\]

The curvature tensors $R$ and $R'$ of $\n$ and $\n'$ and the
corresponding Ricci tensors $\rho$ and $\rho'$ are related via the
formulae (see \eg \cite{IvPapa, FriIv}):
\begin{eqnarray}\label{23'}
    R(x,y,z,w)=R'(x,y,z,w)&-&\frac{1}{2}\left(\n'_x T\right)(y,z,w)+\frac{1}{2}\left(\n'_y
    T\right)(x,z,w)
 \nonumber \\
       &-&\frac{1}{4}g\left(T(x,y),T(z,w)\right)-\frac{1}{4}\sigma^T\left(x,y,z,w)\right),
\end{eqnarray}
\begin{equation}\label{23'''}
\begin{split}
    \rho(y,z)=\rho'(y,z)-\frac{1}{2}g^{ij}\left(\n'_{e_i} T\right)(y,z,e_j)
       -\frac{1}{4}g^{ij}g\left(T(e_i,y),T(z,e_j)\right),
\end{split}
\end{equation}
where $\sigma^T$ is the 4-form determined by
\begin{equation}\label{23''}
    \sigma^T(x,y,z,w))=\mathop{\s} \limits_{x,y,z} g\left(T(x,y),T(z,w)\right).
\end{equation}
Since $\sigma^T$ and $T$ are forms then the scalar curvatures
$\tau$ and $\tau'$ for $\n$ and $\n'$ satisfy the following
relation
\begin{equation}\label{23''''}
    \tau=\tau'-\frac{1}{4}g^{ij}g^{ks}g\left(T(e_i,e_k),T(e_s,e_j)\right).
\end{equation}

\begin{definition}\label{defn-3.1}
A natural connection with totally skew-symmetric torsion on a
Riemannian almost product manifold $(M,P,g)$ is called a
\emph{RPT-connection}.
\end{definition}

\begin{theorem}\label{thm-3.2}
If a Riemannian almost product manifold $(M,P,g)$ admits a
RPT-connection then $(M,P,g)$ is a Riemannian almost product
$\W_3$-manifold.
\end{theorem}
\begin{proof}
Let $\n'$ is a RPT-connection on a Riemannian almost product
manifold \allowbreak $(M,P,g)$. Since $\n'$ is a natural
connection then the tensor $Q$ of the transformation $\n
\rightarrow\n'$ satisfies \eqref{15}, which implies
\begin{equation}\label{24}
    \mathop{\s} \limits_{x,y,z} F(x,y,z)=\mathop{\s} \limits_{x,y,z}\left\{Q(x,y,Pz)-Q(x,Py,z)\right\}.
\end{equation}
According to \eqref{23}, the tensor $Q$ is also 3-form and then we
have
\begin{equation}\label{25}
    Q(x,y,Pz)=Q(y,Pz,x).
\end{equation}

We apply \eqref{25} to \eqref{24} and obtain the characteristic
condition \eqref{4} for the class $\W_3$.
\end{proof}

From \eqref{15}, \eqref{23} and \eqref{4} we obtain the following
properties of the torsion $T$ of a RPT-connection:
\begin{eqnarray}\label{26'}
    T(x,y,z)&=&T(Px,Py,z)-2F(z,y,Px)\nonumber\\
    &=&T(Px,y,Pz)-2F(y,x,Pz)
    =T(x,Py,Pz)-2F(x,Py,z).
\end{eqnarray}

Let us suppose that $(M,P,g)$ is not a Riemannian $P$-manifold.
The condition $(M,P,g)\notin\W_0$, \thmref{thm-2.2}, \eqref{14'}
and \eqref{26'} imply the following properties:
\begin{equation}\label{26}
  \begin{array}{l}
    p_1(x,y,z)=0,\\
    p_2(x,y,z)=F(z,x,Py)\neq 0,\\
    p_3(x,y,z)=\frac{1}{2}\bigl\{F(x,y,Pz)+F(y,z,Px)-F(z,x,Py)\bigr\}\neq 0,\\
    p_4(x,y,z)=T(x,y,z)-\frac{1}{2}\mathop{\s}
    \limits_{x,y,z}F(x,y,Pz).
  \end{array}
\end{equation}

\begin{theorem}\label{thm-3.4}
On any Riemannian almost product $\W_3$-manifold
$(M,P,g)\notin\W_0$ there exists a unique RPT-connection $\n'$
which torsion is expressed by the tensor $F$. The following
properties are valid for the torsion $T$ of $\n'$:
\[
T\in\mathcal{T}_2\oplus\mathcal{T}_3,\quad
T\notin\mathcal{T}_2,\quad T\notin\mathcal{T}_3,
\]
\begin{equation}\label{21}
        T(x,y,z)=\frac{1}{2}\mathop{\s} \limits_{x,y,z}
        F(x,y,Pz).
\end{equation}
\end{theorem}
\begin{proof}
By \eqref{26} we establish directly that a RPT-connection on a
Riemannian almost product $\W_3$-manifold $(M,P,g)\notin\W_0$ has
a torsion $T=p_2+p_3+p_4$, \ie
$T\in\mathcal{T}_2\oplus\mathcal{T}_3\oplus\mathcal{T}_4$. Let us
suppose that $T=p_2+p_4$. Then \eqref{26} and \eqref{4} imply
$F=0$, \ie $(M,P,g)\in\W_0$ which is a contradiction. Therefore
$T\neq p_2+p_4$ holds. Analogously we establish that $T\neq
p_3+p_4$. Consequently we obtain
$T\notin\mathcal{T}_2\oplus\mathcal{T}_4$ and
$T\notin\mathcal{T}_3\oplus\mathcal{T}_4$. The condition
$T\in\mathcal{T}_2\oplus\mathcal{T}_3$, \ie $T=p_2+p_3$, according
to \eqref{26} implies $p_4=0$, which is \eqref{21}.
\end{proof}

Let $\n'$ is the RPT-connection with torsion $T$ determined by
\eqref{21}. Then \eqref{23}, \eqref{2} and \eqref{4} imply
    \begin{gather}
        Q(x,y,z)=-\frac{1}{4}\left\{F(x,Py,z)-F(Px,y,z)-2F(y,Px,z)\right\},\label{27}\\
    Q(x,y)=-\frac{1}{4}\bigl\{\left(\n_{x} P\right)Py-\left(\n_{Px} P\right)y-2\left(\n_{y} P\right)Px
    \bigr\}.\label{28}
    \end{gather}
Bearing in mind \eqref{23} and \eqref{28}, we have
\begin{equation*}
    T(x,y)=-\frac{1}{2}\bigl\{\left(\n_{x} P\right)Py
    -\left(\n_{Px} P\right)y-2\left(\n_{y} P\right)Px
    \bigr\}.
\end{equation*}
According to \eqref{23''''}, \eqref{7''} and the latter equality,
we obtain the following
\begin{proposition}\label{prop-3.5}
The scalar curvatures $\tau$ and $\tau'$ for the connections $\n$
and $\n'$ on a Riemannian almost product $\W_3$-manifold $(M,P,g)$
are related via the formula
\begin{equation*}\label{27'}
    \tau=\tau'+\frac{3}{8}\nP.
\end{equation*}\hfill$\Box$
\end{proposition}

\begin{corollary}\label{prop-3.6}
The scalar curvatures for the connections $\n$ and $\n'$ on a
Riemannian almost product $\W_3$-manifold $(M,P,g)$ are equal if
and only if $(M,P,g)$ is a Riemannian $P$-manifold. \hfill$\Box$
\end{corollary}

\subsection{A relation of the RPT-connection with other connections}

Now we will find the relation between the RPT-connection $\n'$ and
other two characteristic natural connections on a Riemannian
almost product $\W_3$-manifold.

The canonical connection $\n^C$ on a Riemannian almost product
manifold $(M,P,g)$ is a natural connection introduced in \cite{Mi}
as an analogue of the Hermitian connection in Hermitian geometry
(\cite{Li2, Li1, Gray-Ba-Na-Van}). This connection on
$(M,P,g)\in\W_3$ is studied in \cite{MekDobr}, where for the
tensor $Q^C$ of the transformation $\n \rightarrow\n^C$ it is
obtained
  \begin{equation}\label{28'}
        Q^C(x,y,z)=-\frac{1}{4}\left\{F(y,Px,z)-F(Py,x,z)+2F(x,Py,z)\right\}.
  \end{equation}

The $P$-connection $\n^P$ on a Riemannian almost product manifold
$(M,P,g)$ is a natural connection introduced in \cite{Mek8} as an
analogue of the first canonical connection of Lichnerowicz in
Hermitian geometry (\cite{Gaud, Li, Ya}). In \cite{Mek8} this
connection is studied on $(M,P,g)\in\W_3$, where for the tensor
$Q^P$ of the transformation $\n \rightarrow\n^P$ it is obtained
  \begin{equation}\label{28''}
        Q^P(x,y,z)=-\frac{1}{2}F(x,Py,z).
  \end{equation}

Bearing in mind \eqref{27}, \eqref{28'}, \eqref{28''}, \eqref{3}
and \eqref{4}, we get $Q^P=\frac{1}{2}\left(Q^C+Q\right)$ and
therefore it is valid the following
\begin{proposition}\label{prop-3.7}
The $P$-connection $\n^P$ on a Riemannian almost product
$\W_3$-manifold $(M,P,g)$ is the average connection of the
canonical connection $\n^C$ and the RPT-connection $\n'$, \ie
\[
\n^P=\frac{1}{2}\left(\n^C+\n'\right).
\]
\hfill$\Box$
\end{proposition}


\section{RPT-Connection with Riemannian $P$-Tensor of Curvature}\label{sec-Ptensor}

In the present section we discuss the case when the curvature
tensor $R'$ of the RPT-connection $\n'$ is a Riemannian
$P$-tensor, \ie $R'$ has properties \eqref{9}--\eqref{11}.

Since the torsion $T$ of $\n'$ and the quantity $\sigma^T$ from
\eqref{23''} are forms, then by virtue of \eqref{23'} it follows
that $R'$ satisfies \eqref{9}. According to $\n'P=0$, property
\eqref{11} is valid, too. Then, $R'$ is a Riemannian $P$-tensor if
and only if \eqref{10} holds for $R'$.

Bearing in mind that \eqref{10} is satisfied for $R$, moreover $T$
and $\sigma^T$ are forms, then the following property is valid for
a metric connection $\n'$ with totally skew-symmetric torsion $T$
on a Riemannian manifold $(M,g)$:
\begin{equation}\label{29'}
    \mathop{\s} \limits_{x,y,z} R'(x,y,z,w)=
    \mathop{\s} \limits_{x,y,z}\left(\n'_x T\right)(y,z,w)+\sigma^T(x,y,z,w).
\end{equation}
Then, condition \eqref{10} for $R'$ on a Riemannian almost product
$\W_3$-manifold $(M,P,g)$ holds if and only if
\begin{equation}\label{30'}
    \mathop{\s} \limits_{x,y,z}\left(\n'_x T\right)(y,z,w)+\sigma^T(x,y,z,w)=0.
\end{equation}

\begin{theorem}\label{thm-4.1}
The RPT-connection $\n'$ on a Riemannian almost product
$\W_3$-mani\-fold $(M,P,g)$ has a Riemannian $P$-tensor of
curvature $R'$ if and only if
\begin{equation}\label{31'}
        R(x,y,z,w)=R'(x,y,z,w)-\frac{1}{4}g\left(T(x,y),T(z,w)\right)
        +\frac{1}{12}\sigma^T(x,y,z,w).
\end{equation}
\end{theorem}
\begin{proof}
Let the RPT-connection $\n'$ on $(M,P,g)$ have a Riemannian
$P$-tensor of curvature $R'$, \ie \eqref{30'} is valid. Since $T$
and $\sigma^T$ are forms, then the cyclic sum of \eqref{30'} by
$y$, $z$, $w$ implies
\begin{eqnarray}
    3\left(\n'_x T\right)(y,z,w)&+&2\left(\n'_z T\right)(x,y,w)+2\left(\n'_y
    T\right)(z,x,w)\nonumber\\
    &-&2\left(\n'_w T\right)(x,y,z)+3\sigma^T(x,y,z,w)=0.\label{32'}
\end{eqnarray}

From \eqref{30'} and \eqref{32'} it follows immediately
\[
    \left(\n'_x T\right)(y,z,w)-2\left(\n'_w
    T\right)(x,y,z)+\sigma^T(x,y,z,w)=0,
\]
where by virtue of the substitution $x\leftrightarrow w$ we have
\[
    \left(\n'_w T\right)(x,y,z)-2\left(\n'_x
    T\right)(y,z,w)-\sigma^T(x,y,z,w)=0.
\]
The latter two equalities imply
\begin{equation}\label{33'}
    \left(\n'_x T\right)(y,z,w)=-\frac{1}{3}\sigma^T(x,y,z,w)
\end{equation}
and then, according to \eqref{23'}, we obtain \eqref{31'}.

\emph{Vice versa}, let \eqref{31'} be satisfied. Then, bearing in
mind the equalities $\mathop{\s} \limits_{x,y,z} R(x,y,z,w)=0$ and
$\mathop{\s} \limits_{x,y,z} \sigma^T(x,y,z,w)=3\sigma^T(x,y,z,w)$
for the 4-form $\sigma^T$, we establish that the cyclic sum of
\eqref{31'} by $x$, $y$, $z$ yields $\mathop{\s} \limits_{x,y,z}
R'(x,y,z,w)=0$. Therefore, $R'$ is a Riemannian $P$-tensor.
\end{proof}

Since condition \eqref{33'} is fulfilled for a Riemannian
$P$-tensor $R'$, then we have $g^{ij}\left(\n'_{e_i}
T\right)(y,z,e_j)=0$. Hence, because of \eqref{23'''}, we obtain
\begin{corollary}\label{cor-4.2}
If the RPT-connection $\n'$ on a Riemannian almost product
$\W_3$-mani\-fold $(M,P,g)$ has a Riemannian $P$-tensor of
curvature $R'$, then the Ricci tensors $\rho$ and $\rho'$ for $\n$
and $\n'$ are related via the formula
\[
        \rho(y,z)=\rho'(y,z)-\frac{1}{4}g^{ij}g\left(T(e_i,y),T(z,e_j)\right).
\] \hfill$\Box$
\end{corollary}


\section{RPT-Connection with Riemannian $P$-Tensor of Curvature and Parallel Torsion}\label{sec-parT}

Let $\n'$ be a metric connection with totally skew-symmetric
torsion $T$ on a Riemannian manifold $(M,g)$. If $T$ is parallel,
\ie $\n'T=0$, then \eqref{23'} implies
\begin{equation}\label{35'}
        R(x,y,z,w)=R'(x,y,z,w)-\frac{1}{4}g\left(T(x,y),T(z,w)\right)-\frac{1}{4}\sigma^T(x,y,z,w).
\end{equation}
\emph{Vice versa}, if \eqref{35'} is valid, because of \eqref{23'}
and the totally skew-symmetry of $T$, then $\n'T=0$ holds.
Therefore, relation \eqref{35'} is a necessary and sufficient
condition for a parallel torsion of $\n'$.

Because of $\sigma^T(x,y,z,w)=\sigma^T(z,w,x,y)$ and
$R(x,y,z,w)=R(z,w,x,y)$, relation \eqref{35'} yields
\begin{equation}\label{36'}
        R'(x,y,z,w)=R'(z,w,x,y).
\end{equation}

Besides, if $\n'T=0$ then \eqref{29'} implies
\begin{equation}\label{37'}
    \mathop{\s} \limits_{x,y,z} R'(x,y,z,w)=\sigma^T(x,y,z,w).
\end{equation}

Now, let $\n'$ be the RPT-connection on a Riemannian almost
product $\W_3$-manifold $(M,P,g)$. Let us suppose that the torsion
of $\n'$ is parallel. Then \eqref{36'} is valid. Moreover,
\eqref{11} holds because of $\n'P=0$. Therefore, \eqref{36'} is
equivalent to
\begin{equation}\label{38'}
        R'(Px,Py,Pz,Pw)=R'(x,y,z,w).
\end{equation}
Equalities \eqref{37'} and \eqref{38'} imply
\begin{equation*}\label{39'}
       \sigma^T(Px,Py,Pz,Pw)=\sigma^T(x,y,z,w).
\end{equation*}


In addition to the above, let us suppose that the curvature tensor
$R'$ of the RPT-connection $\n'$ is a Riemannian $P$-tensor. Then
\eqref{33'} is valid and because of $\n'T=0$ the equality
$\sigma^T=0$ holds. Hence and \eqref{35'} we obtain the following
\begin{theorem}\label{thm-5.1}
If the RPT-connection $\n'$ on a Riemannian almost product
$\W_3$-manifold $(M,P,g)$ has a parallel torsion $T$ and a
Riemannian $P$-tensor of curvature $R'$ then
\[
        R(x,y,z,w)=R'(x,y,z,w)-\frac{1}{4}g\left(T(x,y),T(z,w)\right).
\] \hfill$\Box$
\end{theorem}

\section{An Example}\label{sec-exa}

In the present section we study the RPT-connection $\n'$ on an
example of a Riemannian almost product $\W_3$-manifold constructed
in \cite{MekDobr} by a Lie group.

\subsection{The Riemannian almost product $\W_3$-manifold $(G,P,g)$}

We will describe briefly the example of a Riemannian almost
product $\W_3$-manifold in \cite{MekDobr}.

Let $G$ be a 4-dimensional real connected Lie group and $\g$ be
its Lie algebra with a basis $\{X_i\}$. By an introduction of a
structure $P$ and a left invariant metric $g$ as follows
\begin{equation*}\label{45}
    PX_1=X_3,\quad PX_2=X_4,\quad PX_3=X_{1},\quad PX_4=X_{2},
\end{equation*}
\begin{equation}\label{46}
    g(X_i,X_j)=
\begin{cases}
\begin{array}{rl}
1, \quad & i=j;\\
0, \quad & i\neq j,
\end{array}
\end{cases}
\end{equation}
the manifold   $(G,P,g)$ becomes a Riemannian almost product
manifold with $\tr{P}=0$.

The setting of the condition for the associated metric $\tilde{g}$
to be a Killing metric \cite{Hel}, \ie
\begin{equation}\label{48}
    g\left([X_i,X_j],PX_k\right)+g\left([X_i,X_k],PX_j\right)=0,
\end{equation}
specializes $(G,P,g)$ as a Riemannian almost product
$\W_3$-manifold and the Lie algebra $\g$ is determined by the
following equalities:
\begin{equation}\label{50}
\begin{array}{ll}
    [X_1,X_2]= \lm_1 X_1 +\lm_2 X_2,\qquad & [X_1,X_3]= \lm_3
    X_2-\lm_1 X_4,\\[4pt]
    [X_1,X_4]= -\lm_3 X_1 -\lm_2 X_4,\qquad & [X_2,X_3]= \lm_4
    X_2+\lm_1 X_3,\\[4pt]
    [X_2,X_4]= -\lm_4 X_1 +\lm_2 X_3,\qquad & [X_3,X_4]= \lm_3
    X_3
    +\lm_4 X_4,
\end{array}
\end{equation}
where $\lm_1$, $\lm_2$, $\lm_3$, $\lm_4\in\R$.

Further, $(G,P,g)$ will stand for the Riemannian almost product
$\W_3$-manifold determined by  conditions \eqref{50}.

In \cite{MekDobr} there are computed the components with respect
to the basis $\{X_i\}$ of the tensor $F$, the Levi-Civita
connection $\n$, the covariant derivative $\n P$, the curvature
tensor $R$ of $\n$.

There are computed also the scalar curvature $\tau$ for $\n$ and
the square norm of $\n P$:
\begin{equation}\label{*}
    \tau=-\frac{5}{2}\left(\lm_1^2+\lm_2^2+\lm_3^2+\lm_4^2\right),\qquad
    \nP=4\left(\lm_1^2+\lm_2^2+\lm_3^2+\lm_4^2\right).
\end{equation}

\subsection{The RPT-connection $\n'$ on $(G,P,g)$}

Further in our considerations we exclude the trivial case
$\lm_1=\lm_2=\lm_3=\lm_4=0$, \ie the case of a Riemannian
$P$-manifold $(G,P,g)$.

From \propref{prop-3.5} and \eqref{*} we obtain the following
\begin{proposition}\label{prop-6.3}
The manifold $(G,P,g)$ has negative scalar curvatures with respect
to the Levi-Civita connection $\n$ and the RPT-connection $\n'$,
as
\[
\tau'=-4\left(\lm_1^2+\lm_2^2+\lm_3^2+\lm_4^2\right).
\]
\hfill$\Box$
\end{proposition}

The following proposition is useful for computation of the
components of the RPT-connection and its torsion.
\begin{proposition}\label{prop-6.4}
The RPT-connection $\n'$ on $(G,P,g)$ and its torsion $T$ satisfy
the equalities:
\begin{gather}
    T(X_i,X_j)=-[PX_i,PX_j]\label{53'}\\
    \n'_{X_i} X_j=[X_i,X_j]+P[X_i,PX_j].\label{56}
\end{gather}
\end{proposition}
\begin{proof}
Because of \eqref{46} the following equality is valid
\begin{equation}\label{47}
    2g\left(\n_{X_i} X_j,X_k
    \right)=g\left([X_i,X_j],X_k\right)+g\left([X_k,X_i],X_j\right)+g\left([X_k,X_j],X_i\right).
\end{equation}

Using \eqref{47} and \eqref{48}, we obtain
\[
    2g\left(\n_{X_i} X_j,X_k
    \right)=g\left([X_i,X_j],X_k\right)+g\left(P[X_i,PX_j],X_k\right)-g\left(P[PX_i,X_j],X_k\right).
\]
Then we have
\begin{equation}\label{51}
   2 \n_{X_i} X_j
    =[X_i,X_j]+P[X_i,PX_j]-P[PX_i,X_j].
\end{equation}
The equality \eqref{51} implies
\[
\begin{array}{l}
    2\n_{X_i} PX_j=[X_i,PX_j]+P[X_i,X_j]-P[PX_i,PX_j],\\
    2P\n_{X_i} X_j=P[X_i,X_j]+[X_i,PX_j]-[PX_i,X_j].\\
\end{array}
\]
We subtract the last two equalities, apply the formula for the
covariant derivation of $P$ and obtain
\begin{equation}\label{52}
       2 \left(\n_{X_i} P\right)X_j=-P[PX_i,PX_j]+[PX_i,X_j].
\end{equation}

Let $\n'$ be the RPT-connection on $(G,P,g)$. By virtue of
\eqref{23}, \eqref{28} and \eqref{52}, we have
\begin{eqnarray}
        4T(X_i,X_j,X_k)=-g\bigl(3[PX_i,PX_j]&-&P[X_i,PX_j] \nonumber\\
        &-&P[PX_i,X_j]-[X_i,X_j],X_k\bigr).\label{53}
\end{eqnarray}

According to \eqref{50}, we have
\begin{equation}\label{55}
    [PX_i,PX_j]+P[PX_i,X_j]+P[X_i,PX_j]+[X_i,X_j]=0.
\end{equation}
Bearing in mind \eqref{53} and \eqref{55}, it follows \eqref{53'}.

From \eqref{7'}, \eqref{23}, \eqref{51}, \eqref{55} and
\eqref{53'}, we obtain \eqref{56}.
\end{proof}

Using \eqref{53'} and \eqref{50}, we obtain the following non-zero
components $T_{ijk}=T(X_i,X_j,\allowbreak X_k)$ of $T$:
\begin{equation}\label{54}
    T_{134}=-\lm_1,\quad T_{234}=-\lm_2,\quad T_{123}=-\lm_3,\quad T_{124}=-\lm_4.
\end{equation}
The rest of the non-zero components are obtained from \eqref{54}
by the totally skew-symmetry of $T$.

Combining \eqref{50} and  \eqref{56}, we get the components
$\n'_{X_i} X_j$ of the RPT-connection  $\n'$:
\begin{equation}\label{57}
    \begin{array}{l}
        \n'_{X_1} X_1=-\n'_{X_3} X_3= -\lm_1 X_2 +\lm_3 X_4,\\
        \n'_{X_2} X_2=-\n'_{X_4} X_4=  \lm_2 X_1 -\lm_4 X_3,\\
        \n'_{X_1} X_2=-\n'_{X_3} X_4=  \lm_1 X_1 -\lm_3 X_3,\\
        \n'_{X_1} X_3=-\n'_{X_3} X_1= \lm_3 X_2 -\lm_1 X_4,\\
        \n'_{X_1} X_4=-\n'_{X_3} X_2= -\lm_3 X_1 +\lm_1 X_3,\\
        \n'_{X_2} X_1=-\n'_{X_4} X_3= -\lm_2 X_2 +\lm_4 X_4,\\
        \n'_{X_2} X_3=-\n'_{X_4} X_1= \lm_4 X_2 -\lm_2 X_4,\\
        \n'_{X_2} X_4=-\n'_{X_4} X_2= -\lm_4 X_1 +\lm_2 X_3.\\
\end{array}
\end{equation}

By virtue of \eqref{50} and \eqref{57}, we get the following
non-zero components $R'_{ijks}=R'(X_i,X_j,\allowbreak X_k,X_s)$ of
the tensor $R'$:
\begin{equation}\label{R'ijkl}
\begin{array}{c}
\begin{array}{rr}
R'_{1212}=R'_{1234}=\lm_1^2+\lm_2^2,&\qquad
R'_{1432}=R'_{1414}=\lm_2^2+\lm_3^2,
\\
R'_{3412}=R'_{3434}=\lm_3^2+\lm_4^2,&\qquad
R'_{2341}=R'_{2323}=\lm_1^2+\lm_4^2,
\end{array}
\\
\begin{array}{rr}
R'_{1314}=R'_{1332}=R'_{2441}=R'_{2423}=\lm_1\lm_2-\lm_3\lm_4,\\
R'_{1321}=R'_{1343}=R'_{2421}=R'_{2443}=\lm_1\lm_4-\lm_2\lm_3,\\
\end{array}
\\
\begin{array}{rr}
R'_{1241}=R'_{1223}=R'_{2312}=R'_{2334}=R'_{1421}
\phantom{=\lm_1\lm_3+\lm_2\lm_4..}
\\
=R'_{1443}=R'_{3441}=R'_{3423} =\lm_1\lm_3+\lm_2\lm_4,
\end{array}
\end{array}
\end{equation}
The rest of the non-zero components are obtained from
\eqref{R'ijkl} by the properties $R'_{ijks}=R'_{ksij}$ and
$R'_{ijks}=-R'_{jiks}=-R'_{ijsk}$.

Using \eqref{54} and \eqref{57}, we compute the following non-zero
components
$\left(\n'_iT\right)_{jks}=\left(\n'_{X_i}T\right)\allowbreak
(X_j,X_k,X_s)$ of $\n'T$:
\begin{equation}\label{NTijks}
    \begin{array}{c}
\left(\n'_1 T\right)_{132}=\left(\n'_2 T\right)_{142}
=\left(\n'_3 T\right)_{134}=\left(\n'_4 T\right)_{234}%
=\lm_1\lm_4-\lm_2\lm_3,\\
\left(\n'_1 T\right)_{143}=\left(\n'_2 T\right)_{234}
=\left(\n'_3 T\right)_{123}=\left(\n'_4 T\right)_{142}%
=\lm_1\lm_2-\lm_3\lm_4,\\
\left(\n'_1 T\right)_{234}=\left(\n'_3
T\right)_{142}=\lm_1^2-\lm_3^2,\qquad
\left(\n'_2 T\right)_{143}=\left(\n'_4
T\right)_{123}=\lm_2^2-\lm_4^2.
    \end{array}
\end{equation}
The rest of the non-zero components are obtained from
\eqref{NTijks} by the totally skew-symmetry of $T$.

The exterior derivative $\dd T$ of the totally skew-symmetric
 torsion $T$ for the connection $\n'$ on a Riemannian manifold
 $(M,g)$
is given by the following formula (see \eg \cite{IvPapa})
\[
\dd T(x,y,z,w) = \mathop{\s} \limits_{x,y,z}\left(\n'_x
T\right)(y,z,w)-\left(\n'_w T\right)(x,y,z)+2\sigma^T(x,y,z,w).
\]
Then, using \eqref{54},  \eqref{NTijks} and \eqref{23''}, we
compute that $\dd T(X_i,X_j,X_k,X_s)=0$ for all $i,j,k,s$.
Therefore, the manifold $(G,P,g)$ has a closed torsion 3-form $T$
for the RPT-connection $\n'$. In this case we say that $(G,P,g)$
has a \emph{strong} RPT-structure by analogy with the Hermitian
case. Therefore, we have the following
\begin{proposition}\label{prop-6.5}
The manifold $(G,P,g)$ has a strong RPT-structure. \hfill$\Box$
\end{proposition}

In the following theorem we prove that the necessary and
sufficient condition $R'$ to be a Riemannian $P$-tensor and $T$ to
be parallel is one and the same.

\begin{theorem}\label{thm-7.7}
Let $\n'$ be the RPT-connection with curvature tensor $R'$ and
torsion $T$ on the manifold $(G,P,g)$. Then the following
conditions are equivalent:
\begin{romanlist}[(iii)]
    \item $R'$ is a Riemannian $P$-tensor;
    \item $T$ is parallel;
    \item $\lm_3=\varepsilon\lm_1$, $\lm_4=\varepsilon\lm_2$, $\varepsilon=\pm 1$.
\end{romanlist}
\end{theorem}
\begin{proof}
As we have commented in Sec.~\ref{sec-Ptensor}, the RPT-connection
$\n'$ has a Riemannian $P$-tensor of curvature $R'$ if and only if
\eqref{10} is valid for $R'$. By virtue of \eqref{R'ijkl} we
establish directly that $R'$ satisfies \eqref{10} if and only if
conditions (iii) hold, \ie (i) and (iii) are equivalent.

Using \eqref{NTijks}, we obtain immediately the equivalence of
$\n'T=0$  and (iii), \ie (ii) and (iii) are equivalent.
\end{proof}


\end{document}